\documentclass[11pt]{amsart} 

\usepackage{hyperref} 

\usepackage[latin1]{inputenc}
\usepackage{graphicx}
\usepackage{amsmath}
\usepackage{amsthm} 
\usepackage {amsfonts, amssymb} 
\usepackage{version}
\usepackage{float}

\newtheorem{thm}{Theorem}[section]

\newtheorem{prop}[thm]{Proposition}
\newtheorem{cor}[thm]{Corollary}

\theoremstyle{definition}

\theoremstyle{remark}
\newtheorem{oss}{\bfseries Remark}[section]
\newtheorem{ex}{\bfseries Example}[section]

\renewcommand{\Re}{\mathop{\mathrm{Re}}}
\renewcommand{\Im}{\mathop{\mathrm{Im}}}

\newcommand{\ralg}{\bar{\mathfrak g}} 
\newcommand{\lalg}{\mathfrak g}  

\title[Coverings for $4$-dimensional almost complex manifolds ...]
{Coverings for $4$-dimensional almost complex manifolds with non-degenerate torsion}
\author{Cristina Bozzetti \and Costantino Medori}

\address{Cristina Bozzetti:
Dipartimento di Matematica ed Applicazioni, 
Universit\`a di Milano Bicocca,
Via Cozzi 55, 20125, Milano, Italy} 
\email{cristina.bozzetti.83@gmail.com}

\address{Costantino Medori:
Dipartimento di Scienze Matematiche, Fisiche e Informatiche
\\ Universit\`a di Parma
\\ Parco Area delle Scienze 7/a (Campus), 43124 Parma,
Italy} 
\email{costantino.medori@unipr.it}

\date{\today}

\keywords{Almost complex manifold, non-degenerate torsion bundle, absolute parallelism,
infinitesimal automorphism}
\subjclass[2010]{32Q60, 53C15, 57S20, 58A17} 
 
\begin{document}
\maketitle

\begin{abstract}
In \cite{BM} it is shown that almost complex manifolds $(M^4,J)$ of real dimension 4 with non-degenerate torsion bundle admit a double absolute parallelism and it is provided the classification of 
homogeneous $(M^4,J)$ having an associated non-solvable Lie algebra. We extend such a classification to the analysis of the manifolds having an associated solvable Lie algebra, up-to-coverings.
Moreover, for homogeneous $(M^4,J)$ we provide examples with connected and non-connected double covering, thus proving that 
 in general
the double absolute parallelism is not the restriction of two absolute parallelisms.
Furthermore, it is given the definition of a natural metric induced by the absolute parallelisms on $(M^4,J)$ 
and an example of an almost complex manifold with non-degenerate torsion endowed with that metric 
such that it becomes an almost K\"ahler manifold. 
\end{abstract}


\section*{Introduction}
While there exists a local equivalence among complex manifolds, 
almost complex manifolds are not equivalent in general and
the problem of the local equivalence among them is still open. 
However, for a class of almost complex manifolds of dimension 4, having a 
a particular subbundle of their tangent bundle, it is possible to describe all their locally equivalent almost complex structures. 
A partially study about the classification of these kinds of manifolds is given in \cite{BM}. 
In this paper we go on with this study, showing also concrete examples and analyzing the local equivalence among these manifolds.  Homogeneous almost complex manifolds in dimension 6
with semisimple isotropy are classified in \cite{AKW}.

More precisely, in this paper we consider  $4$-dimensional almost complex manifolds $(M^4,J)$ whose Nijenhuis tensor image forms a non-integrable bundle $\mathcal V$, called \emph{torsion bundle}.
For any point $p \in M^4$, there exist two adapted frames of the tangent space $T_pM$, of the form $(X_1, X_2, X_3, X_4)_p$ and $(-X_1, -X_2, X_3, X_4)_p$, which 
provide a $\mathbb Z_2$-structure on $(M^4,J)$. 
Assuming that $(M^4,J)$ is locally homogeneous, we consider the Lie algebra $\lalg$ 
associated to $(M^4,J)$ through the double absolute parallelism.
When $\lalg$ is non-solvable a complete classification (of germs) of $(M^4,J)$ is given in \cite{BM}, while here we analyze also 
the solvable case. 
To this aim, we focus on almost complex manifolds having $\mathbf {A_{4.1}}$ (the unique nilpotent algebra of dimension $4$) and $\mathbf {A_{3.2} \oplus A_1}$ as associated Lie algebras, analyzing and providing examples for all the cases that occur. 

Moreover, we provide  connected homogeneous manifolds with a connected double covering
 (but also with a non-connected one) and we study examples of such manifolds having $\mathbf {A_{4.1}}$ as associated Lie algebra.
 In particular, we obtain that in general the double absolute parallelism is not the restriction of two absolute parallelisms.
 
Finally, we define a metric on $(M^4,J)$, induced by its adapted frames, so that $(M^4,J)$ becomes an almost 
K\"ahler manifold and we also give an example of such a manifold. 
We prove that when $(M^4,J)$ is homogeneous, it is not possible to give on
it a structure of almost K\"ahler manifold.
\smallskip\par
The paper is structured in the following way. 
In section \ref{section1} we recall the basic notions and  properties about almost complex manifolds, and we focus on 
that having  dimension 4 with non-degenerate torsion bundle. 
Here, we highlight the role of the image of the Nijenhuis tensor in the construction of the torsion bundle $\mathcal V$, dwelling on the properties of the spaces of the filtration 
$\mathcal V^+_p \subseteq \mathcal V_p \subseteq \mathcal V_{-2|p} \subseteq \mathcal V_{-3|p} \subseteq T_pM$ and on the existence of two adapted frames on the almost complex manifold $(M^4,J)$ which give rise to a double absolute parallelism on it.
We also recall the definition of a Lie algebra $\lalg$ associated to a connected, locally homogeneous 
almost complex manifold $(M^4,J)$ with non-degenerate torsion and we explain that its 
derived algebra $\lalg'$ must have dimension $2$ or $3$.

In section \ref{section2} we extend the study that we did in \cite{BM}, related to a non-solvable Lie algebra associated to a locally homogeneous almost complex manifold $(M^4,J)$ with non-degenerate torsion bundle, to the study of the solvable case. Here, the classification of the non-equivalent $J$'s is restricted to the study of significant examples because of the big number of solvable Lie algebras. In particular, we analyzed the $4$-dimensional Lie algebras $\mathbf {A_{4.1}}$ and $\mathbf {A_{3.2} \oplus A_1}$ (for the notation and a classification of the solvable Lie algebras, see for example \cite{PBNL}). The former splits into the cases 
where
$\mathcal V_p$ is fundamental or non-fundamental, 
the latter splits into the cases 
where
the space $\mathcal V_p \cap \lalg'$ is trivial or not. 
We provide concrete examples of these kind of manifolds and 
we show that there exist manifolds with non-degenerate torsion 
of the form $M^4 = M^3 \times \mathbb R$ 
for which $M^3$ is locally equivalent to a $3$-dimensional hypersphere, but also for which $M^3$ is not.

In section \ref{section3} we assume that $(M^4,J)$ is a $4$-dimensional connected, homogeneous almost complex manifold with non degenerate torsion bundle. 
When the double covering $F$ of $(M^4,J)$ is connected,
it 
is isomorphic to the connected component of the group of the automorphisms $Aut(M,J)$ of $(M^4,J)$;
whereas, when $F$ is non-connected, the connected component of $Aut(M,J)$ is isomorphic to $M$ (see Proposition \ref{connessioneF}).
Furthermore, we provide examples of homogeneous manifolds $(M^4,J)$ with connected  $F$
(but also with a non-connected one), showing that 
in general it is not possible to have two independent absolute parallelisms globally, but only locally.
Moreover, in Proposition \ref{MconA4.1} we give two representatives of the classes of equivalence of 
homogeneous manifolds $(M^4,J)$ having $\mathbf {A_{4.1}}$ as associated Lie algebra.

Section \ref{section4} is devoted to the construction of metrics, invariant and not, 
which are compatible with the almost complex structure $J$ of a manifold of dimension $4$ with non-degenerate torsion bundle. We give the definition of a natural metric induced by the absolute parallelisms on $(M^4,J)$ 
and we provide an example of a manifold endowed with such a metric
so that it becomes an almost K\"ahler manifold. 
We show that when the manifold is homogeneous there is not an invariant metric on it.
This section provides just some preliminary results that could be a starting point for further research.

\section{Preliminaries}\label{section1}

An \emph{almost complex manifold} $(M^{2n},J)$ is a $2n$-differentiable manifold $M$ endowed with an \emph{almost complex structure} $J$, that is, an endomorphism $J:TM \rightarrow TM$ of the tangent bundle $TM$ such that $J^2=-id$.

An almost complex structure $J$ is a \emph{complex structure} if and only if it is \emph{integrable} 
that is the \emph{Nijenhuis tensor} $N_J$,
which is the skew-symmetric (1,2)-tensor defined by
$$N_J(X,Y):=[JX,JY] -[X,Y] - J([JX,Y]+[X,JY]),$$
vanishes $\forall X,Y \in \Gamma(TM)$.
Note that
$$N_J(X,JY)=-JN_J(X,Y)=N_J(JX,Y).$$

If $(M,J)$ and $(M',J')$ are two almost complex manifolds,  
a differential mapping $F:M \rightarrow M'$ \ is called \emph{(J,J')-holomorphic}, or briefly \emph{holomorphic}, if its differential satisfies $dF \circ J = J' \circ dF.$
We say that $F$ is \emph{(J,J')-biholomorphic} when $F$ is also a diffeomorphism. 

We recall that two almost complex manifolds $(M,J)$ and $(M',J')$ are \emph{locally equivalent} at points $p \in M$ and $p'\in M'$, if there exist a neighborhood $U$ of $p$, a neighborhood $U'$ of $p'$ and a $(J,J')$-biholomorphic map $\varphi: U \rightarrow U'$,
such that $\varphi(p)=p'$ and $\varphi_{*}(J)=J'$.

For an almost complex manifold $(M^{2n},J)$, 
an \emph{infinitesimal automorphism} is a tangent vector field $V \in \Gamma(TM)$ that satisfies 
$$
[V,JX]=J[V,X], \qquad \forall X \in \Gamma(TM).
$$
The \emph{symmetry algebra} $aut_p(M^{2n},J)$, for any $p \in M^{2n}$, is the set of germs of infinitesimal automorphisms at $p$. 
We denote with $aut^0_p(M^{2n},J)$ the \emph{isotropy algebra} at $p$, that is 
the set of infinitesimal automorphisms 
vanishing in $p$.

For any $p \in M^{2n}$, we denote  
$$ \mathcal V_p=\{X \in T_pM : X=N_J(A,B) \mbox{ for some } A, B \in T_pM\}$$
as the set of the images of the Nijenhuis tensor in $p$. 
From the properties of $N_J$, we obtain that $\mathcal V_p$ is $J$-invariant and that it is an even dimensional space. 
When the rank of $\mathcal V_p$ is constant $\forall p \in M^{2n}$, 
we have that
$$\mathcal{V}:= \bigcup_{p \in M} \mathcal{V}_p$$
forms a bundle that we call \emph{torsion bundle}.
We note that for a complex manifold $(M^{2n},J)$ we have $\mathcal V_p = \{0\}$
and $\mathcal V$ is the trivial bundle.
\medskip\par
If we consider the almost complex manifolds $(M^{4},J)$ of 
\emph{real dimension $4$}, 
the torsion bundle becomes very relevant to their analysis and their classification. In \cite{BM} we provide a classification (up to covering) of homogeneous $4$-dimensional almost complex manifolds with non-degenerate torsion bundle with a non-solvable Lie algebra of infinitesimal automorphisms. 
In this paper, we go on with the study of these manifolds 
and we proceed with their analysis
when their associated Lie algebra is solvable.
We recall here some definitions and results obtained in \cite{BM} useful to this aim. 
\medskip\par
Let $\mathcal V$ be the torsion bundle, we say that $\mathcal V$ is \emph{non-degenerate at $p$} $\in M^{4}$ 
if we have
$$[X,JX]_p \notin \mathcal V_p,$$
for some $X \in \Gamma(\mathcal V)$.
The torsion bundle $\mathcal V$ is called \emph{non-degenerate} if it is non-degenerate at any $p \in M^{4}$.
When $\mathcal V$ is non-degenerate at a point $p$, 
then $\mathcal V$ is non-degenerate in a neighborhood of $p$, because of the smoothness of $N_J$.

Let $(M^4,J)$ be an almost complex manifold of real dimension 4, with $N_J \neq 0$ at $p$ and such that 
$\mathcal V$ is non-degenerate at $p$. Then we set
$$
\mathcal V_{-1|p}:= \mathcal V_{p}, \qquad
\mathcal V_{-2|p}:= \mathcal V_{p} + [\Gamma (\mathcal V),\Gamma (\mathcal V)]_p \neq \mathcal V_{p}, 
$$
$$
\mathcal V_{-3|p}:= \mathcal V_{-2|p} +[\Gamma (\mathcal V),[\Gamma (\mathcal V),\Gamma (\mathcal V)]]_{p}.
$$
For any section $A$ of $\mathcal V_{-1}$ such that $A_p \neq 0$, the linear application 
 $$\tau_{T_p^A} :
\begin{cases}
\mathcal V_p &  \rightarrow \mathcal V_p  \\
X_p & \mapsto N(X_p,T_p^A ),
\end{cases}
$$
where $T_p^A:=[A,JA]_p$ is a fixed vector,
has exactly two eigenspaces, that we denote with $\mathcal V_p^{+}$ and $\mathcal V_p^{-}$, which do not depend on the choice of the field $A$ and such that $J\mathcal V_p^{+}  = \mathcal V_p^{-}$.
We also denote with  
$$\mathcal V^{\pm}=\bigcup_{p \in M}\mathcal V_p^{\pm}$$ 
the fiber bundles obtained as the union of these eigenspaces, for $p \in M$.
We obtain the following filtration of the tangent space $T_pM$:  
$$\mathcal V_{p}^+ \subsetneqq \mathcal V_p =\mathcal V_{-1|p} \subsetneqq \mathcal V_{-2|p} \subseteq \mathcal V_{-3|p}  \subseteq T_pM,$$
and we say that $\mathcal V_{p}$ is \emph{fundamental} when $\mathcal V_{-3|p} = T_pM$ and \emph{non-fundamental} when $\mathcal V_{-2|p} = \mathcal V_{-3|p}$.

We recall some properties of these spaces. 
If $X \in \Gamma(\mathcal V)$ is any vector field such that $X_{p} \neq 0$, we have:
\begin{enumerate}
\item  $(X_{p},JX_{p})$ gives a base of $\mathcal V_p$;
\item there exists a neighborhood $U$ of $p$ such that $\forall q \in U$ the dimension of $\mathcal V_q$ is constant and is 2;
\item  $(M^4, \mathcal V, J_{|\mathcal V})$ is a CR manifold and in particular $N_J(A,B)=0$, $\forall A,B \in \mathcal V_p$;
\item $(X_p,JX_p,T^X_p)$ is a base of $\mathcal V_{-2|p}$, where $T^X_p:=[X,JX]_p$;
\item $(X_p,JX_p,T^X_p,JT^X_p)$ is a base of $T_pM$.
\end{enumerate}

Here we report some  results obtained in \cite{BM} that are significant in the next analysis.  

\begin{prop}	\label{distinguished-field}
A \emph{distinguished section} $X$ is locally determined
in $\mathcal V^+$;
such a section is unique up to its sign and it is such that 
$\tau_{X_q}$ has eigenvalue $1$, for all $q$ in an open neighborhood $U$ of $p \in M$. 
\end{prop}

We will call \emph{distinguished field} of $\mathcal V^+$ such a section $X$.
We note that if $X$ is the distinguished section of $\mathcal V^+$, then $JX$ is the distinguished section of 
$\mathcal V^-$. 
For any point $p \in M$, the isomorphisms
$$f_p :
\begin{cases}
(\mathbb R^4,J) &  \rightarrow (T_pM,J) \\
(e_1,e_2,e_3,e_4) & \mapsto (X_p, JX_p, T_p^X, JT_p^X) \\ 
\end{cases}
$$
are called \emph{adapted frames} at a point $p$, 
where $(e_1,e_2,e_3,e_4)$ is the canonical base of $\mathbb R^4$ and $(X_p, JX_p, T_p^X, JT^X_p)$ is a base of $T_pM$, with $X$ 
the distinguished section of $\mathcal V^+$.

\begin{thm}\label{adaptedframes}
If $(M^4,J)$ is an almost complex manifold with dimension $4$ and with non-degenerate torsion bundle $\mathcal V$, 
then, for each point  $p \in M$, there exists an unique pair of adapted frames
$$
\begin{array}{rll}
f'_p (e_1,e_2,e_3,e_4)  & = & (X_p, JX_p, T_p^X, JT^X_p),\\
f''_p (e_1,e_2,e_3,e_4) & = & (-X_p, -JX_p, T_p^X, JT^X_p),
\end{array}
$$
where $X$ is one of the two distinguished sections of $\mathcal V^+$ in a neighborhood of $p$. 
\end{thm}

\begin{prop}\label{Z_2}
The set of all adapted frames $f_p$, with $p$ in $M$, forms a reduction $F$ of the principal bundle of 
linear frames $L(M)$ on $M$ which has structure group isomorphic to $\mathbb Z_2$.
\end{prop}

The two adapted frames $f'_p$ and  $f''_p$ of Theorem \ref{adaptedframes}
give a $\mathbb Z_2$-structure on $(M^4,J)$, while 
the isomorphism
$(\pi_{*|f})^{-1} \circ f : \mathbb R^4 \rightarrow T_{f}F$
is an adapted frame for $F$ that
gives an $\{e\}$-structure on $F$, where 
$$\pi: F \rightarrow M$$ 
is the canonical projection of $F$ on $M$
and $f$ is any point $f_p (e_1,e_2,e_3,e_4)$ in $F$.
Therefore, there is an absolute parallelism on $F$. 

We provide $F$ with the almost complex structure given by the lift of $J$ to $F$,
which we denote  
with the same letter.

If we denote with $Aut(M,J)$ the group of the automorphisms of an almost complex manifold $(M,J)$, we have
\begin{thm}\label{autoM}
The group of automorphisms $Aut(M,J)$ of $(M^4,J)$ is a Lie group and it has dimension
$\dim Aut(M,J) \leq 4$
and so the symmetry algebra of $(M^4,J)$ at $p$ also has dimension
$\dim aut_p(M,J) \leq 4.$
In particular, the group of automorphisms $Aut^0(M,J)$ of $(M^4,J)$ that fix a point $p$ has at most two elements and 
$ aut^0_p(M,J)=\{0\}.$
\end{thm}

Now we are going to recall the construction of the Lie algebra associated to a locally homogeneous almost complex manifold of dimension 4 with a non-degenerate torsion bundle $\mathcal V$.

If $(X_1, X_2, X_3, X_4)$ is an adapted frame of $(M^4,J)$ and $c_{ij}^k$ are the local constant structure defined by $[X_i,X_j]= c_{ij}^k X_k,$ with $i,j=1, ..., 4$, 
we define $\lalg$ as the Lie algebra generated by $\{ X_1, X_2, X_3, X_4\}$.
We will denote with $\lalg^{(k)}$ its derived algebras and with $\lalg^k$ its descending central series; 
since $\lalg^{(1)}=\lalg^{1}$, for simplicity we will denote both with $\lalg '$. 
We also denote by $\mathfrak {z}(\lalg)$ the \emph{center} of $\lalg$
and we note that $\mathfrak {z}(\lalg) \cap \mathcal V_p = \{0\}.$
 
\begin{prop}\label{dimg'}
If $(M^4,J)$ is a locally homogeneous almost complex manifold with non-degenerate torsion bundle, 
then the derived algebra $\lalg'$ of the Lie algebra $\lalg$ associated to $(M^4,J)$ has dimension
$ 2 \leq \dim \lalg ' \leq 3. $
In particular, when $\lalg$ is non-solvable we have that $\dim \lalg '=3$ and that it is the reductive algebra
$\lalg = \lalg ' \oplus \mathfrak z(\lalg).$ 
\end{prop}

\section{Almost complex manifolds of dimension $4$ with solvable $\lalg$}\label{section2}
Now, we consider the classification (up to covering) of homogeneous $4$-dimensional almost complex manifolds with non-degenerate torsion bundle having a solvable associated Lie algebra $\lalg$ (for the classification with a non-solvable $\lalg$ see \cite{BM}).
We restrict our study to the generalization of several significant examples, 
since the complete classification of the almost complex structures of $(M^4,J)$ having solvable $\lalg$ has a very large number of cases and we would risk to become redundant.
Anyway, these generalizations provide the method to make the whole classification when $\lalg$ is solvable. 

\subsection {\large Case $\lalg=\mathbf {A_{4.1}}$}
We start with the classification of the almost complex manifolds having $\mathbf {A_{4.1}}$ as associated Lie algebra since 
$\mathbf {A_{4.1}}$ is the unique nilpotent algebra of dimension 4 which is not decomposable (see e.g. \cite{On3}). 

We recall that if $(e_1,e_2,e_3,e_4)$ is a base of $\mathbf {A_{4.1}}$, 
the Lie algebra $\mathbf {A_{4.1}}$ is defined by $ [e_2,e_4]= e_1$, $[e_3,e_4]= e_2$ and null all the other Lie brackets.
We are going to study the non equivalent almost complex structures $J$ on $(M^4,J)$ 
having the same Lie algebra $\mathbf {A_{4.1}}$.

We have $\lalg' = \langle e_1,e_2\rangle$, $\lalg^{(2)}= 0$ and $\lalg^{2}= \langle  e_1\rangle$. 
Moreover, $\mathfrak {z}(\lalg)=\langle e_1\rangle$ and 
 $\mathcal V_p \neq \lalg' =\langle e_1,e_2\rangle$, otherwise $\mathcal V_p$ would be degenerate. 
So, there are two possibilities: 
$\dim (\mathcal V_p  \cap \lalg' )$ is $1$ or $0$. 

\begin {prop} \label{generKimLee}
If $\lalg$ is the Lie algebra $\mathbf {A_{4.1}}$ associated to a connected locally homogeneous almost complex manifold $(M^4, J)$
with non-degenerate torsion bundle $\mathcal V$, 
then the following facts are equivalent:
\begin{enumerate}
\item [$(a)$]$ \dim (\mathcal V_p  \cap \lalg' )=1$,
\item [$(b)$]$ \mathcal V_{p}$ is non-fundamental,
\item [$(c)$]$ \lalg'  \subseteq \mathcal V_{-2|p}$,
\item [$(d)$]$ \mathcal V_p  \cap \lalg'   = \mathcal V^{\pm}_p$.
\end{enumerate}
\end {prop}

\begin{proof}$(a) \Rightarrow (b)$.
Let us suppose that the $\dim (\mathcal V_p  \cap \lalg' )=1$, then there exist $k,h \in \mathbb{R}$, with $(k,h) \neq (0,0)$, 
such that $\mathcal V_p  \cap \lalg' = \langle ke_1+he_2\rangle$, hence 
$\mathcal V_p =\langle ke_1+he_2, a e_1+b e_2 +c e_3+d e_4\rangle$ 
for some $a,b,c,d \in \mathbb{R}$ (with $(c,d) \neq (0,0)$, otherwise the vector filed $a e_1+b e_2 +c e_3+d e_4$ is in $\lalg' $). 
In this way we have $\mathcal V_{-2|p}=\langle e_1, e_2, ce_3+de_4\rangle$, so that $\mathcal V_{-2|p}$ is a Lie subalgebra of $\lalg$, 
i.e. $\mathcal V_p$ is non-fundamental.\\
$(b) \Rightarrow (c)$. Since $\mathcal V_{p}$ is non-fundamental, $\mathcal V_{-2|p}$ is a 3-dimensional Lie subalgebra of $\lalg$, 
hence there exist $a,b,c,d,x,y,z,t \in \mathbb R$ such that 
$\mathcal V_{-2|p}$  is generated by 
$$
\left\{\!
\begin{array}{l}
\xi:=a e_1+b e_2 +c e_3+d e_4, \\
J\xi:=x e_1+y e_2 +y e_3+t e_4, \\
\eta:=[\xi,J\xi]=(bt-dy)e_1+ (ct-dz)e_2, \\
\end{array}
\right.
$$
with $(ct-dz,bt-dy)\neq(0,0)$ and $(d,t)\neq (0,0)$;
moreover 
$$[\xi, \eta]=-d(ct-dz)e_1 \in \mathcal V_{-2|p}, \qquad [J \xi, \eta] = -t(ct-dz)e_1 \in \mathcal V_{-2|p}.$$ 

If $(ct-dz)(bt-dy)\neq0$, we have $e_1\in \mathcal V_{-2|p}$ and hence $e_2\in \mathcal V_{-2|p}$.
If $(ct-dz)=0$ and $(bt-dy)\neq0$, $e_1 \in \mathcal V_{-2|p}$ and $c e_3+d e_4$ is proportional to $y e_3+t e_4$ 
so that $e_2 \in \mathcal V_{-2|p}$. 
If $(ct-dz)\neq0$ and $(bt-dy)=0$, then $\langle e_1,e_2\rangle \subseteq \mathcal V_{-2|p}$.
In all cases $\lalg'  \subseteq \mathcal V_{-2|p}$. \\
$(c) \Rightarrow (d)$. Since $\langle e_1,e_2\rangle =\lalg' \subseteq \mathcal V_{-2|p}$, there exist $k,h \in \mathbb{R}$, not both null, 
such that $ ke_1+he_2 \in \mathcal V_p $. Let us define 
\begin{equation} \label{NProposiz00}
\left\{\!
\begin{array}{l}
\xi=ke_1+he_2, \quad (k,h) \neq(0,0),  \\
J\xi=a e_1+b e_2 +c e_3+d e_4, \quad  (c,d) \neq(0,0), \\
\eta=[\xi,J\xi]=hde_1, \quad  hd \neq 0, \\
J\eta=	x e_1+y e_2 +z e_3+t e_4. \\
\end{array}
\right.
\end{equation}
We have that $\langle \xi, J\xi \rangle=\mathcal V_p$  and that 
the Nijenhuis tensor of $\xi$ and $\eta$ is 
$$N_J(\xi,\eta)= \dfrac{1}{h}(ct-dz)\xi +\left( -\dfrac{k}{h^2d}(ct-dz) + \dfrac{1}{hd}(bt-dy) \right) \eta -\dfrac{t}{d}J\eta,$$
and since the coefficients of $\eta$ and $J\eta$ must be zero, we obtain $t=0$ and $y= \dfrac{k}{h}z$, 
from which 
\begin{equation} \label{NProposiz}
N_J(\xi,\eta)= - \dfrac{dz}{h}\xi.
\end{equation}
This means that $\xi$ is the distinguished direction that generates $\mathcal V^{+}_p$ or $\mathcal V^{-}_p$ 
(it depends on the sign of the coefficients $d,h,z$ appearing ahead $\xi$ in the calculation of $N_J$).  
We have actually proved $(d)$.\\
$(d) \Rightarrow (a)$ is trivial. 
\end{proof}

We note that since $\mathfrak z(\lalg) \cap \mathcal V_p = \{0\}$ and $e_1 \in \mathfrak z(\lalg)$, the subspace 
$\langle e_1\rangle$ is never contained in $\mathcal V_p$
(in agreement with the fact that $h$, which appears into the proof of Proposition \ref{generKimLee}, is never zero). 

\begin {prop} \label{generKimLee2}
If $\lalg$ is the Lie algebra $\mathbf {A_{4.1}}$ associated to a connected locally homogeneous almost complex manifold $(M^4, J)$
with non-degenerate torsion bundle $\mathcal V$,
then the following facts are equivalent:
\begin{enumerate}
\item [$(a')$]$ \dim (\mathcal V_p  \cap \lalg' )=0$,
\item [$(b')$]$ \mathcal V_{p}$ is fundamental,
\item [$(c')$]$ \lalg'  \nsubseteq \mathcal V_{-2|p}$,
\item [$(d')$]$ \mathcal V_p  \cap \lalg'   = \{0\}$.
\end{enumerate}
\end {prop}


\noindent 
Now we distinguish the cases 
with non-fundamental $\mathcal V_p$ and fundamental $\mathcal V_p$.

\subsubsection{Case $\mathbf {A_{4.1}}$ with non-fundamental $\mathcal V_p$.}\label{A4.1n.f.} 
\

Since for Proposition \ref{generKimLee} we have that $\mathcal V_p  \cap \lalg'   = \mathcal V^{\pm}_p$ when $\mathcal V_p$ is non-fundamental, we can write the distinguished sections of $\mathcal V^{\pm}_p$ in a simpler manner, that is, 
as a linear combination of the base of $\lalg'$. 
For example, when $\mathcal V_p  \cap \lalg'   = \mathcal V^{+}_p$, we obtain the following
\begin{cor}\label{generA_{4.1}nonfund}
If $\mathbf {A_{4.1}}$ is the Lie algebra $\lalg$ associated to a locally homogeneous almost complex manifold  $(M^4, J)$
with non-fundamental and non-degenerate torsion bundle $\mathcal V$ such that $\mathcal V_p  \cap \lalg'   = \mathcal V^{+}_p$ and 
$\xi$ is the distinguished section of $\mathcal V^+_p$, we have 
$$
\left\{\!
\begin{array}{l}
\xi=ke_1-dze_2,  \\
J\xi=a e_1+b e_2 +c e_3+d e_4, \quad d \neq 0, \\
\eta=[\xi,J\xi]=-d^2ze_1, \quad z \neq 0\\
J\eta=	x e_1-\dfrac{k}{d} e_2 +z e_3, \\
\end{array}
\right.
$$
with $k,a,b,c,d,x,z \in \mathbb{R}.$
\end{cor}

\begin{proof}
From the proof of Proposition \ref{generKimLee} we obtain the equation \eqref{NProposiz} for any $\xi \in \mathcal V^+_p$.
Since we want that $\xi$ is the distinguished field in $p$, we have to put $ \dfrac{dz}{h}=-1$ and substitute the result in \eqref{NProposiz00}. 
\end{proof}

In \cite{KL}, Kim and Lee gave an example of a manifold which belongs to the class of mandifolds 
considered in Proposition \ref{generKimLee} and in Corollary \ref{generA_{4.1}nonfund}.
They found an almost complex manifold with non-degenerate torsion bundle, with non-fundamental $\mathcal V_p$,
having $\mathbf {A_{4.1}}$ as associated Lie algebra 
and with symmetry algebra $aut_p(M,J)$ of dimension 4. We summarize it below.

\begin{ex}[Kim and Lee]\label{KimLee}
If $(z_1,z_2)$ are complex coordinates of $\mathbb{C}^2$ ($z_j= x_j+iy_j$, for $j=1,2$) and 
$$\dfrac{\partial}{\partial z_j}= \dfrac{1}{2}\left(\dfrac{\partial}{\partial x_j} - i \dfrac{\partial}{\partial y_j}\right), \quad 
\dfrac{\partial}{\partial \bar{z}_j}= \dfrac{1}{2}\left(\dfrac{\partial}{\partial x_j} + i \dfrac{\partial}{\partial y_j}\right),$$
the $(1,0)$ vector fields 
$$Z_1= \dfrac{\partial}{\partial z_1} -2\bar z_1 i \dfrac{\partial}{\partial z_2},$$
$$Z_2= (z_1-\bar z_1)\dfrac{\partial}{\partial z_1} +(z_1-\bar z_1)\bar z_1 \dfrac{\partial}{\partial \bar z_1}+
(-2i-z^2_1+\bar z^2_1)\dfrac{\partial}{\partial z_2}+(-z^2_1+\bar z^2_1)\dfrac{\partial}{\partial \bar z_2}, $$
define an almost complex structure $J$ on $\mathbb{R}^4$. 
A base of germs of the infinitesimal automorphisms $V_1,V_2,V_3,V_4$ of $(\mathbb{R}^4, J)$ is given by 
$$
\begin{array}{l}
V_1= Z_2+ \bar Z_2,  \\[2pt]
V_2=  Z_1+ \bar Z_1 - i(z_1  -\bar z_1)V_1,  \\[2pt]
V_3=  i(Z_1- \bar Z_1) - (z_1  +\bar z_1)V_1 - (2z_1 \bar z_1 + z_2  +\bar z_2)V_2,  \\[2pt]
V_4=  i(Z_2- \bar Z_2) - 2i(z_1  -\bar z_1)V_2 - ( z_1  -\bar z_1)^2 V_1. \\
\end{array}
$$
They generate the symmetry algebra $aut_p(\mathbb{R}^4, J)$ for any point $p$. 
\end{ex}

\vspace{5pt}
We are going to analyze the above example in light of our results. 
Computing the distinguished section $X$ of $\mathcal V^+_p$, we obtain that 
the adapted frame of $(\mathbb{R}^4, J)$ is given by: 
$$X=  \dfrac{1}{\sqrt2} \Im Z_1, \quad JX= \dfrac{1}{\sqrt2} \Re Z_1, \quad T= \dfrac{1}{2} \Re Z_2, \quad JT= - \dfrac{1}{2} \Im Z_2$$
where 
$$\Re Z_1=  \dfrac{1}{2}\dfrac{\partial}{\partial x_1} -x_1\dfrac{\partial}{\partial x_2} + y_1 \dfrac{\partial}{\partial y_2}, \quad 
\Im Z_1=  -\dfrac{1}{2}\dfrac{\partial}{\partial y_1} +x_1\dfrac{\partial}{\partial y_2} + y_1 \dfrac{\partial}{\partial x_2}, $$
$$\Re Z_2=  -\dfrac{\partial}{\partial y_2}, \quad
\Im Z_2=  2y_1\dfrac{\partial}{\partial x_1} -(4x_1y_1+1)\dfrac{\partial}{\partial x_2}.$$
Moreover, an easy calculation gives that 
$[ \Re Z_1, \Im Z_1 ]= - \Re Z_2,$  and that $[ \Im Z_1, \Im Z_2 ]= -2 \Re Z_1,$
with null all the other Lie brakets.
Hence, it holds
\begin{equation}\label{KMrule}
[X,JX]=T, \quad [X,JT]=JX,
\end{equation}
which gives exactly $\mathbf {A_{4.1}}$ as Lie algebra $\lalg$ associated to $(\mathbb{R}^4, J)$.
We also note that the torsion bundle $\mathcal V$ is non-degenerate and non-fundamental, 
since $\mathcal V_{-2|p}=\langle X,JX,T\rangle$ is a subalgebra of $\mathbf {A_{4.1}}$.
Moreover, $\lalg'= \langle JX, T\rangle$, the derived algebra $\lalg^{(2)}$ is null 
and the descending central series is $\lalg^{2}= \langle T\rangle$; in particular we have that 
$\mathcal V_p  \cap \lalg'   = \mathcal V^{-}_p$.
Since the manifold $(\mathbb{R}^4, J)$ is homogeneous, its associated Lie algebra $\lalg$
is formed by the left invariant vector fields under the action of the group of the automorphisms on $(\mathbb{R}^4, J)$.
One can check that the infinitesimal automorphisms $V_1,V_2,V_3,V_4$, which generate the symmetry algebra $aut_p(M,J)$, 
generate also the Lie algebra $\ralg$ of the right invariant vector fields
(indeed $[V_j,X_j]=0$, for all $j=1,...,4$,
where the $X_j$ are the fields of the adapted frame) and that 
$$[V_2,V_3]=-2V_1, \qquad [V_3,V_4]=4V_2,$$
and null all the other Lie brackets.
There is an isomorphism from $\lalg$ to $\ralg$ given by
$$\left\{ \!
\begin{array}{ccc}
-2V_1 & \mapsto & T \\
V_2  & \mapsto  & JX \\
V_3  & \mapsto  & -X \\
V_4  & \mapsto  & -4JT. \\
\end{array}\\
\right.
$$

\begin{oss}\label{sphere}
The vector fields $V_1,V_2,V_3$ are the infinitesimal automorphisms of the Heienberg group $H_t= \{ Re z_2 + |z_1|^2=t\}$, 
with $t \in \mathbb R$, which is a 3-dimensional CR manifold. Moreover, $V_1,V_2,V_3$ act transitively on each $H_t$ and $V_4$ is transversal to $H_t$.
So the manifold $(\mathbb R^4,J)$ is homogeneous and it is foliated with spherical hypersurfaces.
\end{oss}

\subsubsection{Case $\mathbf {A_{4.1}}$ with fundamental $\mathcal V_p$.}\

A complete classification in this case involves the use of too many parameters and it would be heavy
(although in principle this can be done). 
We provide only an easy example which is obviously non equivalent to the Kim and Lee one.
We take
$$\xi=e_3, \quad J\xi=e_4, \quad \eta=[\xi,J\xi]=e_2, \quad J\eta=e_1-e_3,$$
so that $N_J(\xi,\eta)= J\xi$.
An easy calculation gives that the adapted frame is 
$$X=\dfrac{1}{\sqrt{2}}(e_3+ e_4),  \quad JX=\dfrac{1}{\sqrt{2}}(e_4- e_3), \quad
T=\dfrac{2}{\sqrt{2}}e_2, \quad
JT=\dfrac{2}{\sqrt{2}}(e_1-e_3). 
$$
Here the almost complex structure $J$ is non equivalent to that of the Kim and Lee example: 
it is easy to check that not only $\mathcal V_p  \cap \lalg^1  = \{0\}$ holds, but $\mathcal V_{p}$ is fundamental too.

\subsection{\large Case $\lalg= \mathbf {A_{3.2} \oplus A_1}$}

In \S.\ref{A4.1n.f.} we analyzed the classification of the manifolds foliated with spherical hypersurfaces and having $\mathbf {A_{4.1}}$ as associated Lie algebra (see Remark \ref{sphere}). 
In this section 
we are interested in finding the existence of 4-dimensional locally homogeneous almost complex manifold 
with non-degenerate torsion bundle and  
with a foliation $M^3 \times \mathbb{R}$ such that 
$M^3$ is non-equivalent to a hypersphere of dimension $3$ (see \cite{Ca2}). 
Our propose is also to find examples of such manifolds and  extend them to a class of equivalent examples. 
To this aim, we consider $\mathbf {A_{3.2} \oplus A_1}$ 
as associated Lie algebra $\lalg$ of $M^3 \times \mathbb{R}$, defined by
\begin{equation}
[e_1,e_3]= e_1, \; [e_2,e_3]= e_1+e_2, \; [e_1,e_2]= 0, \; [e_4,e_j]=0, \quad j=1,2,3,
\end{equation} 
where $\mathbf {A_{3.2}} = \langle e_1,e_2,e_3\rangle$ and $\mathbf {A_{1}} = \langle e_4\rangle$.  
We have $\lalg'=\langle e_1,e_2\rangle$, $\lalg^{(2)}=0$, $\lalg^{2}=\lalg'$ and $\mathfrak{z}(\lalg)=\mathbf {A_1}$. 
Since $\mathcal V_p$ is non-degenerate, we have that $\mathcal V_p \neq \lalg'$, so 
$\dim (\mathcal V_p \cap \lalg') =1$ or $\dim (\mathcal V_p \cap \lalg') =0$. 

\begin{prop}\label{prop-two-cases}
If $\mathbf {A_{3.2} \oplus A_1}$ is the Lie algebra associated to 
a connected locally homogeneous almost complex manifold $(M^4,J)$ with non-degenerate torsion bundle $\mathcal V$,
then the following facts are equivalent:
\begin{enumerate}
\item [$(a)$]$\dim (\mathcal V_p  \cap \lalg')=1,$ 
\item [$(b)$]$\dim (\mathcal V_{-2|p}  \cap \lalg' )=2$.
\end{enumerate}
Moreover, the following facts are equivalent too:
\begin{enumerate}
\item [$(a')$]$\dim (\mathcal V_p  \cap \lalg' )=0,$ 
\item [$(b')$]$\dim (\mathcal V_{-2|p}  \cap \lalg' )=1$.
\end{enumerate}
\end{prop}

\begin{proof}
$(a) \Rightarrow (b)$.
Let us assume, by contradiction, that $\dim (\mathcal V_{-2|p}  \cap \lalg' )=1$, we have that 
$\mathcal V_{-2|p}  \cap \lalg' = \langle ae_1+be_2\rangle$, for certain $a, b \in \mathbb{R}$ with  $(a,b) \neq (0,0)$, and 
$ \mathcal V_{p} = \langle ae_1+be_2, xe_1+ye_2+ze_3+te_4 \rangle$, for certain $x,y,z,t \in \mathbb{R}$. As a consequence, 
$\mathcal V_{-2|p}=\langle e_1,e_2,ze_3+te_4\rangle$, a contradiction.
\\
The proof of $(b)\Rightarrow (a)$  is straightforward (it is sufficient to use Grassmann theorem on the dimension 
of $\mathcal V_{-2|p}  + \lalg'$ considering that $\lalg' \subseteq \mathcal V_{-2|p}$).  
\end{proof}
According to Proposition \ref{prop-two-cases} we consider the following two cases.

\subsubsection{Case $\mathbf {A_{3.2} \oplus A_1}$ with $\dim (\mathcal V_p  \cap \lalg' )=1$.}\

When $\dim (\mathcal V_p  \cap \lalg' )=1$, we define
$$
\left\{ \!
\begin{array}{l}
\xi=ke_1+e_2,  \\
J\xi=a e_1+b e_2 +c e_3+d e_4, \quad c \neq 0, \\
\eta=[\xi,J\xi]=c(k+1)e_1+ce_2, \\
J\eta=	(2ac+(k-1)(y-2bc)) e_1+ye_2 +2c^2 e_3+te_4, \quad t \neq 2cd, \\
\end{array} 
\right.
$$
with $k,a,b,c,d,x,y,t \in \mathbb{R}$. Here we can assume that the coefficient of $e_2$ in the definition of $\xi$ is 1, 
because it is not null: if it was null, $\xi,J\xi$ and $\eta$ would be dependent one each other 
($\xi$ would be proportional to $\eta$); the condition $t \neq 2cd$ also arises from the independence of $\xi,J\xi,\eta,J\eta$.
A calculation gives that 
$$N_J(\xi,\eta) = c(2bc-y) \xi -c^2J \xi.$$
Note that the condition $\dim (\mathcal V_p  \cap \lalg' )=1$ implies that $\mathcal V_p$ is non-fundamental, but it is just a sufficient condition. 
We provide an example of such a manifold.

\begin{ex}
We give an homogeneous space $M^3 \times \mathbb{R}$ with $M^3$ non-equivalent to a hypersphere and with $\mathbf {A_{3.2} \oplus A_1}$ as associated Lie algebra.
We take 
$$
\begin{array}{l}
Z_1= (1-ix_1+ix_2)\partial_1 -ix_2\partial_2 -i\partial_3-i\partial_4, \\[2pt]
Z_2= -i(x_1+x_2)\partial_1+ (1-ix_2)\partial_2-i\partial_3,
\end{array}
$$
as $(1,0)$-vector fields that define an almost complex structure $J$ on $M^3 \times \mathbb{R}$.
In this way we have that $J$ satisfies 
$Je_1=e_3+e_4, \; Je_2=e_3,$
where 
$e_1=\partial_1,$ \,
$e_2=\partial_2,$ \,
$e_3=(x_1+x_2)\partial_1+x_2\partial_2+\partial_3,$ \,
$e_4=\partial_4.$
We have $\mathcal V_p=\langle e_2,e_3\rangle$, so that $\mathcal V_p$ is non-fundamental and non-degenerate.
The adapted frame of $M^3 \times \mathbb{R}$ results 
$$
X=\dfrac{1}{\sqrt{2}}(e_2-e_3),\quad
JX=\dfrac{1}{\sqrt{2}}(e_2+e_3),\quad
T=e_1+e_2,\quad
JT=2e_3+e_4,
$$
hence the associated Lie algebra $\lalg$ of $M^3 \times \mathbb{R}$ is $\mathbf {A_{3.2} \oplus A_1}$.
We have that $\mathcal V_p$ is the holomorphic tangent space $H_pM^3$ of $M^3$, that $\mathcal V_{-2|p}=\mathbf {A_{3.2}}$ and that $T_p \mathbb{R}=\mathbf {A_{1}}$.  
If $\lalg$ is the Lie algebra of the left invariant vector fields, the 
Lie algebra $\ralg$ of the right invariant vector fields is generated by
$$
\begin{array}{l}
W_1= \dfrac{1}{2}e^{x_3}(Z_1 +\bar{ Z_1 }), \\[5pt]
W_2= \dfrac{1}{2}x_3e^{2x_3}Z_1 +\dfrac{1}{2}x_3e^{2x_3} \bar{ Z_1 } + \dfrac{1}{2}e^{x_3}Z_2 +\dfrac{1}{2}e^{x_3}\bar{ Z_2 }, \\[5pt]
W_3= -\dfrac{1}{2}(x_1+x_2)Z_1 - \dfrac{1}{2}(x_1+x_2)\bar{ Z_1 }-\dfrac{1}{2}(x_2+i)Z_2 - \dfrac{1}{2}(x_2-i)\bar{ Z_2 }, \\[5pt]
W_4= -\dfrac{1}{2}i(Z_1 - \bar{ Z_1 } -Z_2 + \bar{ Z_2 }).\\ 
\end{array}
$$
\end{ex}
\begin{oss}
Since $\mathcal V_{-2|p}$ is a subalgebra of $\mathbf {A_{3.2} \oplus A_1}$
and $\mathbf {A_1}$ is the center of $\mathbf {A_{3.2} \oplus A_1}$, there is a foliation of $(M^4,J)$ 
such that the action given by the vector fields of
$\mathcal V_{-2|p}$ sends points of a leaf, having $\mathcal V_{-2|p}$ as Lie algebra, in points of the same leaf, 
and the field $e_4$ sends points of a leaf in points of another leaf.  
Note that we have actually obtained a manifold $M^3$ which is a non-degenerate CR manifold,
in particular, $M^3$ is not diffeomorphic to a hypersphere of dimension $3$. 
For more details we refer the reader to \cite{Ca2} p.70.  
\end{oss}

\subsubsection{Case $\mathbf {A_{3.2} \oplus A_1}$ with  $\dim (\mathcal V_p  \cap \lalg' )=0$.}\

When $\dim (\mathcal V_p  \cap \lalg' )=0$, a classification of corresponding manifolds involves a lot of parameters, 
so we just consider some examples to show the existence of such manifolds.
If we take 
$$
X=e_3,  \quad
JX=e_1+ e_4,  \quad
T=[\xi,J\xi]=-e_1, \quad
JT=	-e_2 +e_4, 
$$
we have that $N_J(X,T)=X$, that is, $X$ is the distinguished section of $\mathcal V^+_p$ 
and $(X,JX,T,JT)$ is the adapted frame. 

We note that $\mathcal V_p$ is non-fundamental in this example, but this not holds in general. 
Indeed, we have the following counterexample:  
$$
\xi=e_3,  \quad
J\xi=e_2+ e_4,  \quad
\eta=[\xi,J\xi]=-e_1-e_2, \quad
J\eta=	e_2 +2e_3.
$$
The Nijenhuis tensor is $N_J(\xi,\eta)=-2\xi$, hence $\langle \xi\rangle=\mathcal V^-_p$, and $\mathcal V_p$ is fundamental.


\section{Homogeneous almost complex manifolds of dimension $4$}\label{section3}

In this section we assume that $(M,J)$ is a four dimensional \emph{connected, homogeneous} almost complex manifold with non degenerate torsion bundle, that is, the action of the connected component $G$ of the Lie group $Aut(M,J)$ of the automorphisms of $(M,J)$ on $(M,J)$ is transitive.
As a consequence, we have that there exists a diffeomorphism $M \cong G/G_p,$ where
$G_p$ is the isotropy subgroup of $G$ fixing a point $p$ of $M$.
If $(F, \pi)$ is the double covering of $(M,J)$ given by
\begin{equation} \label{coveringF}
 \pi: F \rightarrow M,
\end{equation}
where we recall that $F$ is endowed with an $\{e\}$-structure and $M$ is endowed with an $\mathbb Z_2$-structure. 
Now, we have two possibilities based on the fact that $F$ is connected or non-connected.

\begin{prop} \label{connessioneF}
Given a connected homogeneous almost complex manifold $(M,J)$ of dimension $4$ with non-degenerate torsion bundle, 
if $(F, \pi)$ is the double covering of $(M,J)$ given by \eqref{coveringF} 
and $G$ is the connected component of $Aut(M,J)$, we have that:
\begin{enumerate}
\item[(i)] when $F$ is connected, then $F \cong G$ and it has subgroups of order $2$; in particular, $M$ is a Lie group if and only if
					 the isotropy subgroup $G_p$ of $G$ is in the center of $G$; 
\item[(ii)] when $F$ is non-connected, then $F \cong G \times \{0,1\}$ and $M \cong G$. 
\end{enumerate} 
\end{prop}

\begin{proof}

(i) Since the action of $G$ on $F$ is an immersion, from the connection of $G$  
we have that $F \cong G$ when $F$ is connected (because the only connected subgroups of $F$ are the trivial ones). 
Hence, from $M \cong G/ G_p$ and  $M \cong F/ \mathbb Z_2$, we obtain $M \cong G/ \mathbb Z_2$. 
This means that in this case there must be subgroups of $G$ of order 2, so we can write $G_p=\langle id_G, h\rangle$, 
with $h^2=id_G$ and $h\neq id_G$.
In general, $M \cong G/ G_p$ is not a Lie group, but just a manifold. 
We have that $M \cong G/ G_p$ is a Lie group if and only if either $g^{-1}hg=id_G$ or $g^{-1}hg=h$ holds, 
for all $g \in G$. Since $h \neq id_G$, the latter holds, i.e. $G_p \subseteq Z(G)$, where $Z(G)$ is the center of $G$.
So, $G/ G_p$ is a Lie group if and only if $G_p \subseteq Z(G)$. 

(ii) When $F$ is not connected, every sheet of $F$ is isomorphic to $G$ and, since $F \cong G \times \{0,1\}$, we have that $M \cong G$.
In this case the isotropy subgroup of $G$ must be the trivial one. 
\end{proof}


Now, according to Proposition \ref{connessioneF}, we consider the following two cases. 

\subsection{Manifolds $(M,J)$ with non-connected double covering}

We take $M$ as any homogeneous almost complex manifold which is isomorphic to the Lie group
$$ \tilde{G}= \left\{
\left( 
\begin{array}{cccc}
1	& a	& \frac{a^2}{2}		& d \\
0	& 1	& a	 								& c \\
0	& 0	& 1	 								& b \\
0	& 0	& 0	 								& 1 \\
\end{array}\right)
: \quad a,b,c,d \in \mathbb{R} \right \},
$$
having $\mathbf {A_{4.1}}$ as associated Lie algebra.
The group $\tilde{G}$ is connected and simply connected since it is isomorphic (as topologic manifold) to $\mathbb R^4$; 
in particular, it is equivalent to the manifold of Kim and Lee example (see Example \ref{KimLee}).
The double covering $F$ of $M$ is not connected.
Moreover, $\tilde{G}$ is a nilpotent group  whose center is
$$Z( \tilde{G})= \left\{
\left( 
\begin{array}{cccc}
1	& 0	& 0		& \alpha \\
0	& 1	& 0	 	& 0 \\
0	& 0	& 1	 	& 0 \\
0	& 0	& 0	 	& 1 \\
\end{array}\right)
: \quad \alpha \in \mathbb{R} \right \}.
$$

\subsection{Manifolds $(M,J)$ with connected double covering}

Let us consider the group $\tilde{G}$ built above. A normal discrete subgroups of $\tilde{G}$ is  given by
$$ N_t= \left\{
\left( 
\begin{array}{cccc}
1	& 0	& 0		& tk \\
0	& 1	& 0	 	& 0 \\
0	& 0	& 1	 	& 0 \\
0	& 0	& 0	 	& 1 \\
\end{array}\right)
: \quad k \in \mathbb{Z} \right \},
$$
for a fixed $t \in \mathbb R$.
The group $\tilde{G}/N_2$ is connected and is a double covering of $\tilde{G}/N_1$. 
Moreover, 
if $[N_1]$ and $[N_2]$ are the classes of equivalence of the elements $N_1$ and $N_2$ of $\tilde{G}/N_2$, 
we have that $\langle[N_2],[N_1]\rangle$ is a normal subgroup of order two of $\tilde{G}/N_2$, as expected.

\smallskip
Since the fact that $\tilde{G}$ is connected and simply connected implies that 
all the other Lie groups having $\mathbf {A_{4.1}}$ as Lie algebra 
are of the form $\bar{G}\cong \tilde{G}/N$, 
where $N$ is a normal discrete subgroup contained into the center of $\tilde{G}$, 
we have the following result.

\begin{prop}\label{MconA4.1}
Any homogeneous almost complex manifold $(M,J)$, with non-degenerate torsion bundle 
and with the connected component of the group of automorphisms $Aut(M,J)$ isomorphic to $\mathbf {A_{4.1}}$, 
is equivalent to one of these Lie groups: 
\begin{enumerate}
 \item [(i)]
$ \tilde{G}= \left\{
\left( 
\begin{array}{cccc}
1	& a	& \frac{a^2}{2}		& d \\
0	& 1	& a	 								& c \\
0	& 0	& 1	 								& b \\
0	& 0	& 0	 								& 1 \\
\end{array}\right)
: \quad a,b,c,d \in \mathbb{R} \right \},
$\\
which has a non-connected double covering $F \cong \tilde{G} \times \{0,1\}$;  
\item[(ii)]
 $\tilde{G}/N_t,$ where
$ N_t= \left\{
\left( 
\begin{array}{cccc}
1	& 0	& 0		& tk \\
0	& 1	& 0	 	& 0 \\
0	& 0	& 1	 	& 0 \\
0	& 0	& 0	 	& 1 \\
\end{array}\right)
: \quad k \in \mathbb{Z} \right \},
$\\
which has a connected double covering $F \cong \tilde{G}/N_{2t}$. 
\end{enumerate}
\end{prop}
 
\begin{oss}
When the double covering $(F,\pi)$ of any homogeneous $(M^4,J)$ has the non-connected $F$, 
there are two independent absolute parallelisms globally on $(M^4,J)$.
Whereas, when the double covering $(F,\pi)$ of any homogeneous $(M^4,J)$ has the connected $F$, 
we do not have a global absolute parallelism on $(M^4,J)$ given by adapted frames.
\end{oss}

\section{Metrics on almost complex manifolds 
with torsion
}\label{section4}
In the previous sections 
we developed the theory about the existence of two adapted frames on 
almost complex manifolds $(M^4,J)$ of real dimension 4 with non-degenerate torsion bundle.
In this section we  endow these almost complex manifolds with metrics induced by such adapted frames.

First of all we define a \emph{natural metric} on $(M^4,J)$ taking 
$$(X_p, JX_p, T_p^X, JT^X_p)$$
as orthonormal base of $T_pM$, where $p \in M$ and 
$X$ is the distinguished section of $\mathcal V^+$ in a neighborhood of $p$ (see Proposition \ref{distinguished-field}).
Since a $J$-biholomorphic transformation 
sends adapted frames into adapted frames, we have:
\begin{prop}
To fix a frame $(X_p, JX_p, T_p^X, JT^X_p)$, with $X$ the distinguished section of $\mathcal V^+$ and $p \in M$, 
is equivalent to have an invariant canonical norm on $T_pM$.  
In particular, the $J$-holomorphic transformations are isometries of $M$. 
\end{prop}

More in general we could consider a metric on $(M^4,J)$ that it becomes an almost K\"ahler manifold.

Let us consider $(M^4,J)$ endowed with a 
Riemannian metric $\mathcal G$ such that $\mathcal G(X,Y)=\mathcal G(JX,JY)$, for any $X,Y \in \Gamma(TM)$, 
and with a fundamental 2-form $\Omega$, defined by
$\Omega(X,Y)= \mathcal G(X,JY),$ in a way that $(M^4,J,\mathcal G)$ becomes an almost K\"ahler manifold, that is $d\Omega=0$.

We require that $(X_p,JX_p,T_p,JT_p)$ is an orthogonal base of the tangent space $T_pM$ of $(M^4,J,\mathcal G)$ and
$$ \mathcal G (X,X)=a, \quad \mathcal G (JX,JX)=b, \quad \mathcal G (T,T)=c, \quad \mathcal G (JT,JT)=d,$$
for some real differentiable functions $a,b,c,d$.
\\
The $J$-invariance of $\mathcal G$ implies $a=b$ and $c=d$. 
By the formula 
$$
\begin{array}{ll}
d\Omega(X,Y,Z) = & \dfrac{1}{3} \left\{ X \Omega(Y,Z)+Y \Omega(Z,X)+Z \Omega(X,Y) \right. \\ 
                 & \left. - \Omega([X,Y],Z) - \Omega([Z,X],Y) - \Omega([Y,Z],X)\right\}
\end{array}
$$
of $d\Omega$, we obtain that the condition $d\Omega=0$ is equivalent to the system
\begin{equation}\label{metric_condition}
\left\{
\begin{array}{l}
T(a)-\mathcal G ([T,X],X) + \mathcal G([JX,T],JX)=0 \\
JT(a)-c-\mathcal G ([JT,X],X) + \mathcal G([JX,JT],JX)=0 \\
X(c)-\mathcal G ([X,T],T) + \mathcal G([JT,X],JT)+ \mathcal G([T,JT],JX)=0 \\
JX(c)-\mathcal G ([JX,T],T) + \mathcal G([JT,JX],JT)- \mathcal G([T,JT],X)=0. \\
\end{array}
\right.
\end{equation}

\begin{oss}
Any almost complex manifold of dimension 4 has the local symplectic property (see \cite{Le}), i.e., 
given an almost complex structure $J$ on $M^4$ there exists a symplectic form which is compatible with $J$ 
in a neighborhood of each point of $(M^4,J)$, that is $\Omega(JX,JY)=\Omega(X,Y)$, for any $X,Y \in \Gamma(TM)$.
\end{oss}

\begin{oss}
A 4-dimensional almost complex manifold $(M^4,J)$ with non-degenerate torsion is of type 0 
(see Theorem 2.1 in \cite{Mu}), while a 4-dimensional nearly K\"ahler manifold should be of type 1
(see Theorem 4.2 in \cite{Mu}). Then $(M^4,J)$ cannot be nearly K\"ahler.
\end{oss}

\begin{oss}
When $(M^4,J)$ is homogeneous, it is not possible to give on it an invariant structure of almost K\"ahler manifold. 
Indeed, since any $J$-biholomorphic transformation of $(M^4,J)$ sends an adapted frame into an adaped frame, it is an isometry, hence it 
preserves the almost complex structure $J$ and the metric $\mathcal G$, so it also preserve the curvature. 
In \cite{Bl} Blair shows that in dimension 4 there are no almost K\"ahler manifolds of constant
curvature unless the constant is 0, in which case the manifold is K\"ahlerian.
Therefore, an homogeneous $(M^4,J)$ must be a (non-almost) K\"ahler manifold. 
\end{oss}


Finally we analyze the example of Kim and Lee in order to find some solutions of \eqref{metric_condition}. 
For the adapted frame $(X_p,JX_p,T_p,JT_p)$ of $(M^4,J)$ we have
$\mathcal G (X,X)=\mathcal G (JX,JX)=a$ and $\mathcal G (T,T)=\mathcal G (JT,JT)=c.$
If we take $a$ and $c$ as constants, the functions $T(a),JT(a),X(c),JX(c)$ in \eqref{metric_condition} are zero, 
moreover since $[X,JX]=T$ and $[X,JT]=JX,$ (see \eqref{KMrule})
the system \eqref{metric_condition} becomes
$$
\left\{
\begin{array}{l}
\mathcal G (0,X) + \mathcal G(0,JX)=0 \\
-c-\mathcal G (-JX,X) + \mathcal G(0,JX)=0 \\
-\mathcal G (0,T) + \mathcal G(-JX,JT)+ \mathcal G(0,JX)=0 \\
-\mathcal G (0,T) + \mathcal G(0,JT)- \mathcal G(0,X)=0. \\
\end{array}
\right.
$$
This gives $c=0$ for any $a$. 
Therefore, when $a$ and $c$ are constants we cannot define a metric which is compatible with $J$ on $(M^4,J)$. 

There are different results when $a$ and $c$ are differentiable functions. The system \eqref{metric_condition} becomes
$$
\left\{
\begin{array}{l}
T(a)=0,\\
JT(a)-c=0, \\
X(c)=0, \\
JX(c)=0, \\
\end{array}
\right.
$$
$$
\left\{
\begin{array}{l}
\dfrac{\partial}{\partial y_2}(a)=0, \\[10pt]
\dfrac{1}{2}\left( 2y_1\dfrac{\partial}{\partial x_1}-(4x_1y_1+1)\dfrac{\partial}{\partial x_2}\right)(a)+c=0, \\[10pt]
\left( \dfrac{1}{2}\dfrac{\partial}{\partial x_1}-x_1\dfrac{\partial}{\partial x_2}+y_1\dfrac{\partial}{\partial y_2}\right)(c)=0, \\[10pt]
\left( -\dfrac{1}{2}\dfrac{\partial}{\partial y_1}+x_1\dfrac{\partial}{\partial y_2}+y_1\dfrac{\partial}{\partial x_2}\right)(c)=0. \\[10pt]
\end{array}
\right.
$$
A solution of this system is 
\begin{equation} \label{solutionKLmetric}
\left\{
\begin{array}{l}
a=e^{x_1^2+x_2+y_1^2}, \\
c=\frac{1}{2}e^{x_1^2+x_2+y_1^2}. \\
\end{array}
\right.
\end{equation}
Therefore, when $a$ and $c$ are differentiable functions such that the system \eqref{solutionKLmetric} holds, 
$(M^4,J,\mathcal G)$ is an almost K\"ahler manifold.


\end{document}